\documentclass{article}

\usepackage[a4paper, total={6in, 8in}]{geometry}

\usepackage{lipsum}

\usepackage[authoryear]{natbib}
\bibpunct{(}{)}{,}{a}{,}{;} 

\usepackage[english]{babel}
\usepackage[utf8]{inputenc}
\usepackage[T1]{fontenc}

\usepackage{float}

\usepackage[pdftex]{graphicx}

\usepackage{bm}
\usepackage{amsmath}
\usepackage{amsfonts}
\usepackage{amsthm} 

\usepackage[ruled,vlined,resetcount]{algorithm2e}
\usepackage{algorithmic}
\SetKw{KwFrom}{from}
\SetKwHangingKw{Point}{\raisebox{0.25ex}{\scalebox{0.5}{$\blacksquare$}}}
\SetKwInOut{Param}{Parameters}
\SetKw{Return}{Return}

\usepackage{mathrsfs} 

\usepackage[justification=centering]{caption}
\usepackage{subcaption} 

\usepackage{colortbl}
\usepackage[dvipsnames]{xcolor}

\usepackage{authblk}

\usepackage[title]{appendix}

\usepackage[pdftex,colorlinks=true,linkcolor=blue,citecolor=blue,urlcolor=blue]{hyperref}
\usepackage[nameinlink]{cleveref}

\newcommand{\Diag}{\mathrm{Diag}}

\newcommand{\trace}{\mathop{\text{Trace}}}

\newcommand{\DD}{\mathrm{D}}
\newcommand{\TT}{\mathrm{T}}

\newcommand{\q}[1]{``#1''}%
\newcommand{\R}{\mathbb{R}} 
\newcommand{\N}{\mathbb{N}} 

\newcommand{\e}{\mathbb{E}} 
\newcommand{\cov}{\mathrm{Cov}} 

\newtheorem{proposition}{Proposition}[section]

\newcommand{\mc}[1]{\mathcal{#1}}%

\crefname{equation}{eq.}{eqs.}

	\title{Geostatistics for large datasets on Riemannian manifolds: a matrix-free approach}

	\author[1,2]{Mike Pereira%
		\thanks{Corresponding author: \texttt{mike.pereira@minesparis.psl.eu}}}  
	\author[2]{Nicolas Desassis}
	\author[3]{Denis Allard}
	
	\affil[1]{Department of Mathematical Sciences, Chalmers University of Technology and University of Gothenburg \\
		Gothenburg, Sweden}
	\affil[2]{Mines Paris, PSL University, Centre for geosciences and
geoengineering, Fontainebleau, France}
	\affil[3]{Biostatistics and Spatial Processes (BioSP), INRAE, Avignon, France}

\date{October $27^{\text{th}}$, 2022}
\begin{document}

\maketitle

\begin{abstract}
	Large or very large spatial (and spatio-temporal) datasets have become common place in many environmental and climate studies. These data are often collected in non-Euclidean spaces (such as the planet Earth) and they often present nonstationary anisotropies.  This paper proposes a generic approach to model Gaussian Random Fields (GRFs) on compact Riemannian manifolds that bridges the gap between existing works on nonstationary GRFs and random fields on manifolds. This approach can be applied to any smooth compact manifolds, and in particular to any compact surface. By defining a Riemannian metric that accounts for the preferential directions of correlation, our approach yields an interpretation of the non stationary geometric anisotropies as resulting from “local” deformations of the domain. We provide scalable algorithms for the estimation of the parameters and for optimal prediction by kriging and simulation able to tackle very large grids. Stationary and nonstationary illustrations are provided.
\end{abstract}




\section{Introduction}
\label{sec:intro}    

Large or very large spatial (and spatio-temporal) datasets have become common place in many environmental and climate studies. Various approaches have been proposed by the statistical community to tackle the \q{big $N$} challenge in ways that properly acknowledge the spatial and spatio-temporal dependence structures usually observed in these datasets. Recently, several competitions have been organized to compare methods and algorithms in this context \citep{heaton2019case,huang2021competition} which provide an excellent overview of state-of-the-art methods for analyzing large spatial datasets. Most algorithms rely on approximation methods of Gaussian Processes, also known as Gaussian Random Fields (GRFs). In their conclusion, the organizers of the first competition note that \q{spatial data may exhibit anisotropy, nonstationarity, large and small range spatial dependence as well}.
Despite its long history, modeling nonstationary spatial datasets remains a challenge. \citet{sampson1992nonparametric} proposed a space deformation approach further developed for instance in \citet{perrin2000reducing} and \citet{fouedjio2015estimation}. The main difficulty with this approach is to estimate a valid, global, deformation of the domain, which in practice is not guaranteed to exist. \citet{paciorek2006spatial} (see also \citet{fouedjio2016generalized}) introduced a class of nonstationary covariance functions based on the kernel convolution approach of \citet{higdon1999non}, later generalized in \citet{porcu2009quasi}. Our approach rather builds on the so-called \q{SPDE approach} introduced in the seminal work of \citet{lindgren2011explicit}, which relates GRFs characterized with Matérn covariance functions to stationary solutions of a specific Stochastic Partial Differential Equation (SPDE). This approach can be extended to the nonstationary case by allowing spatially varying coefficients of the SPDE, see for instance \citet{fuglstad2015exploring,fuglstad2015does} among other possible references.
Another challenge often faced when analyzing environmental or climate data is to work on non-Euclidean domains. In particular, methods for analyzing data on spheres has received a lot of attention, see \citet{marinucci2011random}, \citet{jeong2017spherical} and \citet{porcu202130} for recent reviews and \citet{rayner2020eustace} for a recent application using the SPDE approach dealing with extremely large datasets. Methods developed for spheres depend usually on specific properties, such as expansion into the spherical harmonics \citep{Emery2019simulating,lang2015isotropic,Lantuejoul2019} or the use of arc distances to define valid covariance models \citep{gneiting2013strictly,Huang2011}. 

This paper aims at bridging the gap between existing works on nonstationary GRFs and random fields on manifolds. Specifically, we propose a generic approach to model GRFs on compact Riemannian manifolds and we provide scalable algorithms for their optimal prediction by kriging and simulation.
Our approach is based on two main ingredients. First, random fields are  defined through expansions in the eigenfunctions of the Laplace--Beltrami operator on the Riemannian manifold which are, in some cases, solutions to some particular SPDE. Then, we build finite element approximation of these GRFs. This construction allows to perform optimal prediction, simulation (including conditional) and estimation of the parameters using scalable algorithms.

For this purpose, we define a Riemannian metric that accounts for the preferential directions of correlation of the nonstationary GRF. This method yields an interpretation of the  \q{local anisotropies} as resulting from \q{local} deformations of the domain, in striking contrast to both the space deformation and the kernel convolution approaches. The resulting fields can be seen as a direct generalization of the construction for nonstationary random fields proposed in \citet{fuglstad2015exploring}.

Our approach can be applied to any smooth compact manifold, and in particular to any compact surface or hypersurface. It shares clear similarities with the work of \citet{borovitskiy2020matern}, but in contrast,  is not restricted to Whittle--Matérn fields since in our approach the GRF is characterized by its spectral density, whose inverse is restricted to belong to the family of positive polynomials. Besides, our approach does not rely on the explicit computation of the eigenfunctions and eigenvalues of the Laplace--Beltrami operator, and can be seen as providing a theoretical motivation to the method developed in \citet{pmlr-v130-borovitskiy21a} to deal with Gaussian processes on graphs.

The flexibility of our approach does not result in increased computational costs. Indeed, we show how prediction and conditional simulations can be performed through a so-called \q{matrix-free} approach. This approach, unlike classical geostatistical algorithms, does not require to build and store possibly large covariance matrices, but instead relies only on products between some sparse matrices and vectors. This in turn ensures the scalability of this method, thus paving the way to efficient nonstationary geostatistics for large datasets. We illustrate our approach with 2D and 3D synthetic examples and grids with more that $10^7$ nodes for the 3D cases.

The organization of the paper is the following. The GRF model is presented in \Cref{sec:rf_approx}, along with its finite element approximation and covariance function. Kriging and simulation algorithms are provided in  \Cref{sec:krig}. Section \ref{sec:estim} shows how the parameters can be estimated using maximum likelihood. All methods are summarized as Algorithms.  Stationary and nonstationary illustrations are then provided in \Cref{sec:illustration}. We conclude with some final words in \Cref{sec:disc}. Proofs and technical details are deferred to the Appendix. Throughout this paper, vectors and matrices will be denoted in bold fonts. The superscript $^T$ denotes the transpose operation on matrices or vectors. $\Diag$ is the operator that transforms a vector of length $n$ into an $n \times n$ matrix whose diagonal elements are those of the vector and whose off-diagonal elements are 0. $\bm I_p$ is the $p \times p$ identity matrix. $\vert \bm A \vert$ is the determinant of the square matrix $\bm A$ and $\Vert\cdot\Vert$ denotes the Euclidean norm. 


\section{Random Fields on Riemannian Manifolds}
\label{sec:rf_approx}

\subsection{Definition and Finite Element Approximation}\label{sec:def_rf}

A generic approach to define and characterize GRFs on Riemannian manifolds has been proposed in \citet{borovitskiy2020matern} and in \citet{ lang2021galerkin} and is now briefly recalled. Let $\mathcal{D}$ be a smooth compact manifold of dimension $d$ equipped with a Riemannian metric $g$, and let $f:\R_+\rightarrow\R_+$ be a function such that for some $\beta>d/2$, $\vert \lambda^{\beta}f(\lambda)\vert$ is bounded as $\lambda\rightarrow +\infty$. A centered GRF $\mc{Z}$ is constructed on the resulting Riemannian manifold $(\mathcal{D},g)$ through the expansion
\begin{equation}
	\mc{Z}=\sum_{k\in\mathbb{N}} f^{1/2}(\lambda_k)w_k e_k,
	\label{eq:lap_field}
\end{equation}
where $f^{1/2}:\R_+\rightarrow\R$ is a function such that $\left(f^{1/2}\right)^2=f$ on $\R_+$, $\lbrace w_k\rbrace_{k\in\mathbb{N}}$ is a sequence of independent standard Gaussian variables, $\lbrace \lambda_k \rbrace_{k\in\N}$ denote the set of eigenvalues of the Laplace--Beltrami operator $-\Delta$ on $(\mathcal{D},g)$, and $\lbrace e_k \rbrace_{1\le k\le \infty}$ denote the associated eigenfunctions.
In order to get a feeling of what \Cref{eq:lap_field} represents, one could recall that on $\mathbb{R}^d$ the eigenfunctions of the Laplacian are all the members of the uncountable family of functions $\{ e^{-i \langle \bm \omega, \bm x \rangle} : \bm \omega \in \mathbb{R}^d\}$. In this case, the Whittle-Matérn random fields are solution to the SPDE $(\kappa^2 - \Delta)^{\alpha/2} {\cal Z} = {\cal W}$ where $\cal W$ is the white noise process on $\mathbb{R}^d$ and the associated  spectral density is $f^{1/2}(\Vert \bm w\Vert) = (\kappa^2 + \Vert \bm \omega \Vert)^{-\alpha/2}$.

Going back to our definition of GRFs on compact manifolds,  the eigenvalues and eigenfunctions of the Laplace--Beltrami operator are not known in general. 
A first approach, proposed by \cite{borovitskiy2020matern}, consists of computing them numerically by solving (approximately) the corresponding eigenvalue problems.
Instead, we approximate $\mc{Z}$  using a finite element approach as in \citet{lindgren2011explicit} and \citet{lang2021galerkin}, as this method will naturally yield scalable algorithms for prediction and sampling tasks. First, the manifold $\mathcal{D}$ is triangulated using $n$ nodes $\bm s_1, \dots, \bm s_n \in\mathcal{D}$ and a family of compact support approximation functions $\lbrace \psi_i\rbrace_{1\le i\le n}$ is defined over $\mathcal{D}$, where each $\psi_i$ is the piecewise linear function equal to $1$ at the node $\bm s_i$ and $0$ at all the other nodes. Then, $\mc{Z}$ is approximated by a linear combination $Z$ defined as
\begin{equation}
	Z=\sum_{i=1}^{n} z_i \psi_i,
	\label{eq:z_galer}
\end{equation}
where for any $i\in\lbrace 1, \dots, n\rbrace$,  $z_i=Z(\bm s_i)$ is the weight associated with the basis function $\psi_i$.

The weights $\bm Z = (z_1, \dots, z_n)^T$ are chosen so that $Z$ can also be written using the same expansion as the one defining the original field $\mc{Z}$ in~\eqref{eq:lap_field}, but replacing now the eigenfunctions $\lbrace e_k\rbrace_{k\in\mathbb{N}}$ and eigenvalues $\lbrace \lambda_k\rbrace_{k\in\mathbb{N}}$ of $-\Delta$ by those of its Galerkin approximation (see \Cref{appen:galerkin_approx} for more details). This particular choice yields an explicit formula to compute these weights, see  \cref{prop:weights} below. We first introduce $\bm C$ and $\bm F$, the mass and stiffness matrices respectively defined by
\begin{equation}
	[\bm C]_{ij}=(\psi_i, \psi_j), \quad [\bm F]_{ij}=(\nabla\psi_i, \nabla\psi_j), \quad 1\le i,j\le n,
\end{equation}
where $(\cdot, \cdot)$ denotes the $L^2$ inner product on the Riemannian manifold (see \Cref{appen:rm}). Note that since each basis function $\psi_i$ ($1\le i\le n$) is zero for every node of the triangulation except one, the resulting matrices $\bm C$ and $\bm F$ are sparse. Let then  $\sqrt{\bm C}\in\R^{n\times n}$ be a matrix such that $\sqrt{\bm C}(\sqrt{\bm C})^T=\bm C$, and let $\bm S$ be the matrix defined by
\begin{equation}
	\bm S = (\sqrt{\bm C})^{-1} \bm F (\sqrt{\bm C})^{-T}.
	\label{eq:def_S}
\end{equation} 
Note that since $\bm S$ is real, symmetric and positive semi-definite, it is diagonalizable with non-negative eigenvalues. Therefore it can be written as 
\begin{equation*}
	\bm S=\bm V \Diag(\lambda_1, \dots, \lambda_n)\bm V^T,
	\label{eq:eigenS}
\end{equation*}
where $\lambda_1, \dots, \lambda_n$ denote the eigenvalues of $\bm S$ and $\bm V$ is an orthogonal matrix whose columns are eigenvectors of $\bm S$.

\begin{proposition}
	Let $\bm Z$ be the vector of weights as in \Cref{eq:z_galer}. Then, the vector $\bm Z$ is a centered Gaussian vector with covariance matrix $\bm\Sigma$ given by
	\begin{equation}
		\bm\Sigma=(\sqrt{\bm C})^{-T}f(\bm S)(\sqrt{\bm C})^{-1},
		\label{eq:cov_mat_g}
	\end{equation}
	where $f(\bm S)$ is a so-called matrix function, defined from the eigendecomposition of $\bm S$ as
	$$f(\bm S)=\bm V \Diag\big(f(\lambda_1), \dots, f(\lambda_n)\big)\bm V^T.$$ 
	\label{prop:weights}
\end{proposition}	
This result is  shown in \citet{pereiraPhd2019} and \citet{lang2021galerkin}. In the latter reference, a convergence result of the approximation of $\mathcal{Z}$ as the mesh size of the triangulation decreases is provided, thus further justifying this  approach. 

The matrix square-root $\sqrt{\bm C}$ can be computed using matrix functions or through a Cholesky decomposition. In practice however, $\sqrt{\bm C}$  is replaced by the so-called mass lumping approximation defined as the diagonal matrix with entries:
\begin{equation}
	[\sqrt{\bm C}]_{ii} = \sqrt{(\psi_i, 1)}, \quad 1\le i\le n.
	\label{eq:def_masslump}
\end{equation}
To ease the notations, but at the cost of a slight abuse of notation, the mass lumping approximation will also be denoted $\sqrt{\bm C}$ in the remainder of this text. As shown in \citet{lindgren2011explicit}, this approximation comes with negligible effect on the covariance of the resulting random field. It allows to readily have access to the  inverse of the square-root matrix  $\sqrt{\bm C}$ and yields
\begin{equation}
	[\bm S]_{ij}=\frac{(\psi_i, \psi_j)}{\sqrt{(\psi_i, 1)}\sqrt{(\psi_j, 1)}}, \quad 1\le i,j\le n,
	\label{eq:def_masslump_S}
\end{equation}
which ensures that the matrix $\bm S$ is also sparse.

The covariance matrix $\bm \Sigma$ in \Cref{eq:cov_mat_g} involves a matrix function defined above through the eigendecomposition of $\bm S$, which is notoriously computationally expensive. To avoid this, two particular cases can be considered.    If the function $f$ is  approximated by a polynomial, the resulting matrix function becomes a matrix polynomial, which can be computed without involving eigendecompositions. This is the rationale behind the Galerkin--Chebyshev approach proposed in \cite{lang2021galerkin}, where the function $f$ is replaced by its Chebyshev polynomial approximation over an interval containing the eigenvalues of $\bm S$. 

We propose here an alternative approach in the spirit of \citet{lindgren2011explicit} and \citet{rue2005gaussian}.  We assume that $f$ is the inverse of a polynomial $P$ taking positive values over  $\R_+$, i.e. $f=1/P$. Then the resulting precision matrix $\bm Q$ of the weights can be expressed as
\begin{equation}
	\bm Q=\bm\Sigma^{-1}=(\sqrt{\bm C})P(\bm S)(\sqrt{\bm C})^{T},
	\label{eq:prec_mat_cov}
\end{equation}
which again involves matrix polynomial instead of a matrix function. Computing the matrix $\bm Q$ can then be done by summing iterates of the matrix $\bm S$, resulting in a matrix whose sparsity depends on the degree of $P$: the higher the degree of $P$, the less sparse $\bm Q$ is. This approach, which we refer to as a \q{matrix free} approach, will be adopted in the algorithms  presented in \Cref{sec:krig} for prediction  and simulation.

\subsection{Second-order Characterization}

\subsubsection{Stationary and Approximately Stationary Covariance Functions}

Let us consider  the GRF $\mathcal{Z}$ defined by~\Cref{eq:lap_field} on some compact Riemannian manifold $(\mathcal{D}, g)$. It is straightforward to show that its covariance function $C_{\mc{Z}}$ can be written as:
\begin{equation}
	C_{\mc{Z}}(\bm x_1,\bm x_2)=\sum_{k\in\mathbb{N}}f(\lambda_k)e_k(\bm x_1)e_k(\bm x_2), \quad \bm x_1,\bm x_2\in\mathcal{D}.
	\label{eq:cov_lap_field}
\end{equation}
In the particular case where  $(\mathcal{D}, g)$  is an Euclidean domain equipped with the usual metric, \citet{solin2020hilbert} show that $C_{\mc{Z}}$ approximates the covariance of  a random field on $\mathcal{D}$ with radial spectral density $f^{1/2}$ (defined in~\Cref{eq:lap_field}). They also provide a uniform bound on the error between the actual covariance function of $\mc{Z}$ and the covariance function associated with $f^{1/2}$ which shows that the approximation improves as we move further away from the boundary of $\mathcal{D}$. Hence, in this case, $\mathcal{Z}$ approximates an isotropic GRF with covariance $C$ given by
\begin{equation}
	C_{\mc{Z}}(\bm x_1,\bm x_2)\approx C(\Vert \bm x_1-\bm x_2\Vert), \quad \bm x_1, \bm x_2\in\mathcal{D},
	\label{eq:cov_iso}
\end{equation}
where $C$ is the inverse Fourier transform of $\bm w \mapsto f^{1/2}(\Vert \bm w\Vert)$.	  

More generally, expansions similar to~\eqref{eq:cov_lap_field} have been used to characterize the covariance of random fields on (Riemannian) manifolds. For instance, \citet{lang2015isotropic} use it to describe covariance functions of random fields on the sphere (endowed with its usual metric), and show an explicit link between the regularity of the resulting field and the decay of the sequence $\lbrace f^{1/2}(\lambda_k)\rbrace_{k\in\N}$. On general compact Riemannian manifolds, \citet{borovitskiy2020matern} characterize their \q{Matérn Gaussian processes in the sense of Whittle} through covariance functions of the form~\eqref{eq:cov_lap_field}, by taking $f^{1/2}$ to be the spectral density of the usual Matérn covariance function  (i.e. as defined for isotropic random fields on $\R^d$). 
Hence, the field $\mathcal{Z}$ defined by~\eqref{eq:lap_field} with covariance function given by~\eqref{eq:cov_lap_field} can be seen as the counterpart, on the Riemannian manifold $(\mathcal{D}, g)$, of the random fields with radial spectral density $f^{1/2}$ on $\R^d$  and covariance function given by the inverse Fourier transform of $f^{1/2}$. Examples of sampled GRFs with Matérn covariance on different domains are presented in \Cref{fig:matern}.

\begin{figure}[thb]
	\centering
	\raisebox{2em}{\includegraphics[width=0.35\textwidth]{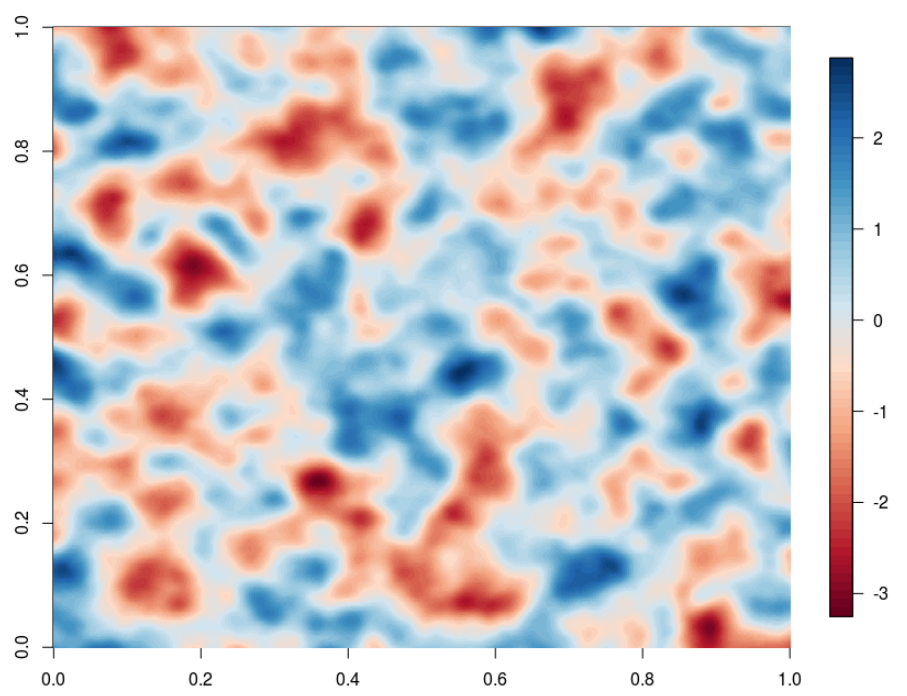}}
	\hspace{2em}
	\includegraphics[width=0.35\textwidth]{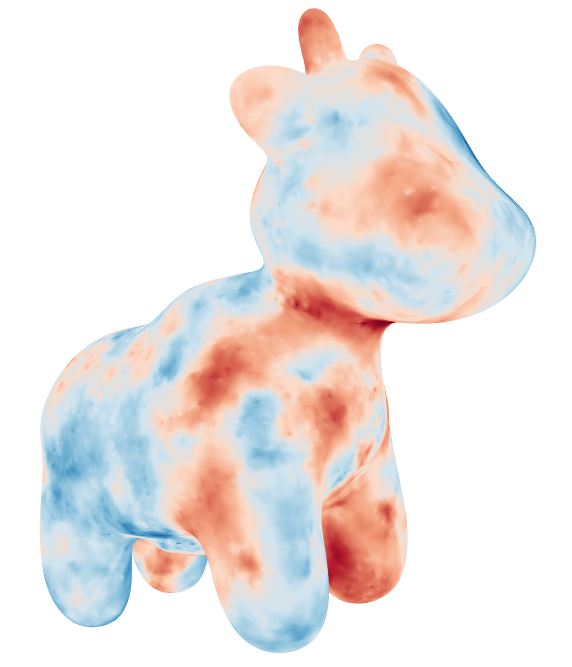} \\
	\includegraphics[width=0.35\textwidth]{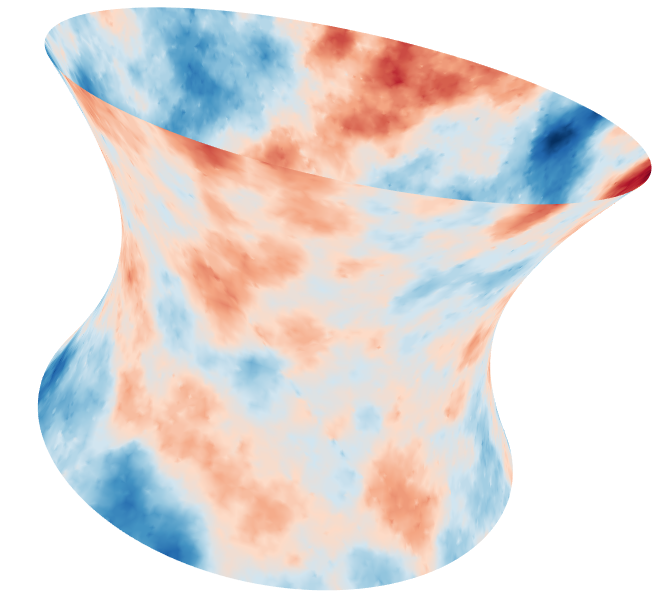}
	\hspace{2em}
	\includegraphics[width=0.35\textwidth]{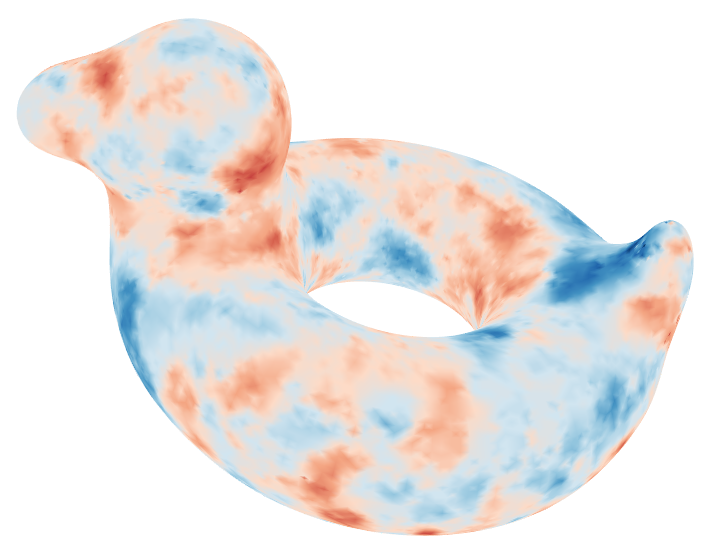}
	\caption{Example of GRFs with Matérn covariance function on the sphere and on surfaces shaped like a cow, a paraboloid and a duck-shaped swim ring.\label{fig:matern}}
\end{figure}

\subsubsection{Nonstationary Covariances}

The general construction of random fields on Riemannian manifold presented in the previous section can be used to define nonstationary models of GRFs, and in particular fields that exhibit \textit{local anisotropies}. 
Such fields are defined on Euclidean domains of dimension $d\in\lbrace 2,3\rbrace$  as follows: around each point of the domain, there is a preferential direction along which the range of highly correlated values is maximal, whereas it is minimal in the orthogonal direction(s). The angles defining the preferential directions are called anisotropy angles and the size of the ranges are called anisotropy ranges. These anisotropy parameters can be graphically represented by an ellipse/ellipsoid whose axes length and direction are respectively given by the anisotropy ranges and angles.

Following the approach described in \citet{pereiraPhd2019}, a GRF with local anisotropies on some bounded Euclidean domain $\mathcal{D}$ can be built by defining a GRF on a specific Riemannian manifold:  anisotropy angles and ranges can be used to define a metric tensor at each point of $\mathcal{D}$. In other words, at each $\bm p \in\mathcal{D}$, the metric is chosen so that it locally \q{deforms} $\mathcal{D}$ into a local domain where the anisotropy reduces to isotropy thanks to the composition of a rotation and a scaling that would turn an ellipse/ellipsoid into a circle/sphere (see \Cref{fig:def_ell}). This transformation results in a metric defined as
\begin{equation}
	g_{\bm p}(\bm u,\bm v)
	=\big(\bm D(\bm p)^{-1}\bm R(\bm p)^{-1} \bm u\big)^T\big(\bm D(\bm p)^{-1}\bm R(\bm p)^{-1} \bm v\big), \quad \bm u, \bm v \in\R^d,
	\label{eq:def_met_nstat}
\end{equation}
where $\bm D(\bm p)$ is the diagonal matrix whose entries are the anisotropy ranges at $\bm p \in \mathcal{D}$ and $\bm R(\bm p)$ is the rotation matrix defined from the anisotropy angles at $\bm p$.

\begin{figure}[thb]
	\centering
	\includegraphics[width=0.8\textwidth]{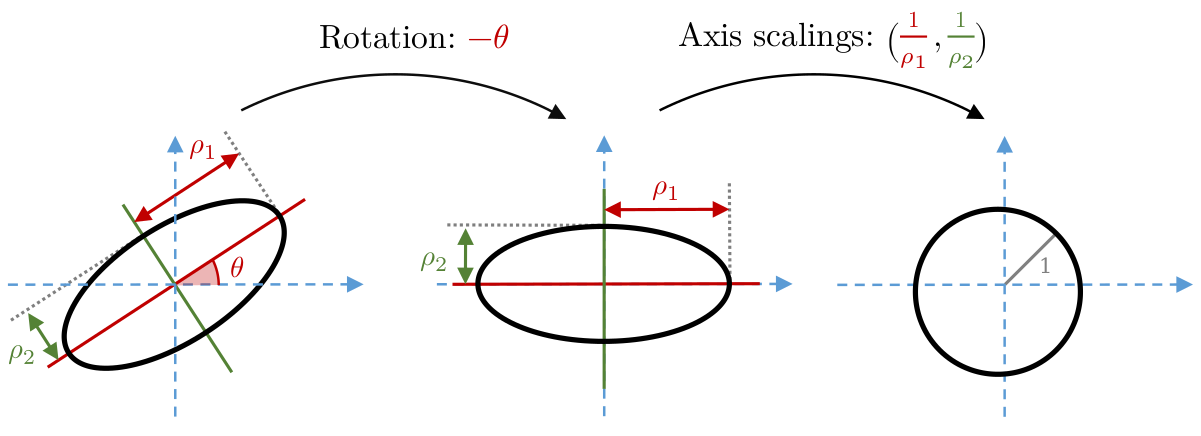}
	\caption{Deformation turning an entropy ellipse (with range parameters $(\rho_1, \rho_2)$ and angle $\theta$) into a circle.}
	\label{fig:def_ell}
\end{figure} 

To get the covariance properties of the field in the original domain $\mathcal{D}$ equipped with the metric~\eqref{eq:def_met_nstat}, we can first apply the deformation and then use~\eqref{eq:cov_iso}, to obtain
\begin{equation}
	C_{\mc{Z}}(\bm p,\bm p+\bm h)\approx C(g_{\bm p}(\bm p,\bm p+\bm h))=C(\Vert\bm D(\bm p)^{-1}\bm R(\bm p)^{-1} \bm h \Vert), 
	\label{eq:cov_nstat_approx}
\end{equation}
where $\bm h \in\R^d$ is some infinitesimal displacement vector around $\bm p$. It is then straightforward to check that such a covariance locally reproduces the desired anisotropy properties around $\bm p$ (see \citet{chiles1999geost} for details). An example of the type of nonstationary fields that can be sampled using this method is presented in \Cref{fig:nstat}.

\begin{figure}[thb]
	\centering
	\includegraphics[height=0.3\textheight]{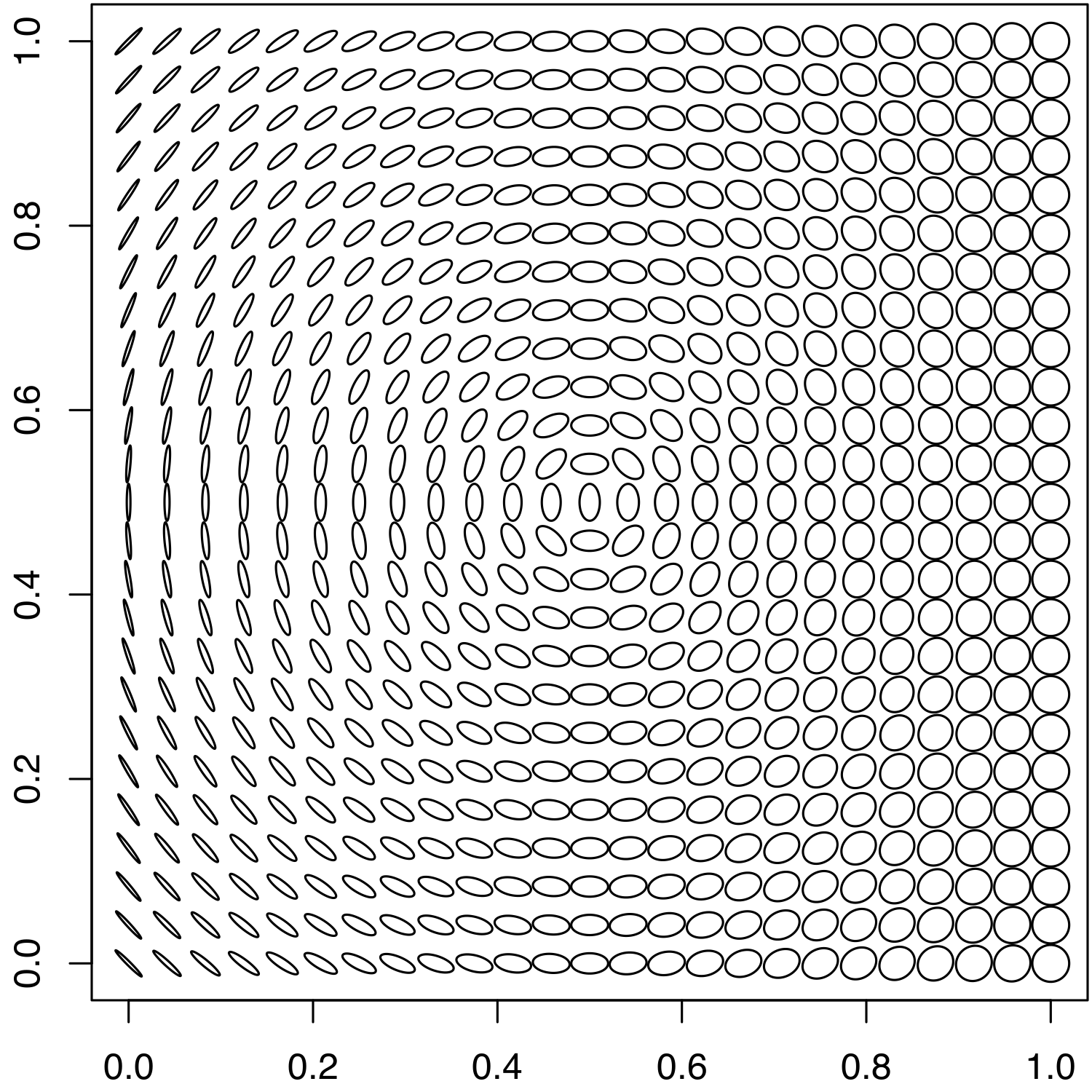}
	\includegraphics[height=0.3\textheight]{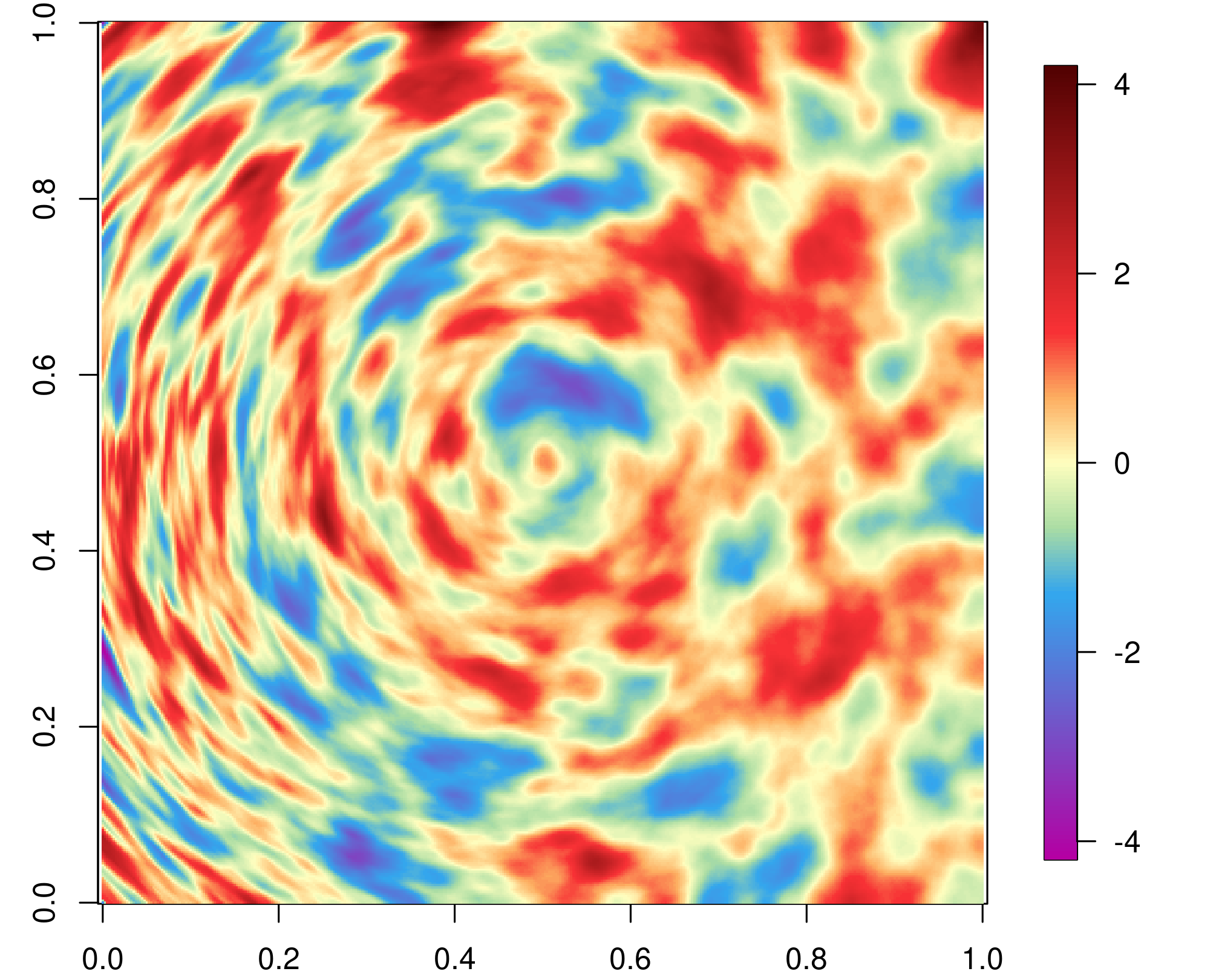}
	\caption{Example of anisotropy parameters (left) and corresponding random field simulation obtained using our method (right), on the unit square.\label{fig:nstat}}
\end{figure}
In conclusion, given a compact Euclidean domain $\mathcal{D}$ and a field of anisotropy parameters on $\mc{D}$, defining nonstationary random fields with the corresponding anisotropy properties can be done by applying the approach described in \Cref{sec:def_rf} to a tailored Riemannian manifold (namely $\mathcal{D}$ equipped with the metric~\eqref{eq:def_met_nstat}). Note that similar ideas could be applied to define random fields with varying covariance structure on more general surfaces if one can define coherent fields of anisotropy parameters on such surfaces.

\section{Prediction on Riemannian Manifolds}\label{sec:krig}

We now show how the construction presented in  \Cref{sec:rf_approx} provides efficient prediction algorithms in a quite general setting, which includes nonstationary covariances, non-Euclidean support and non-Matérn covariance functions.  Given some spatial domain $\mathcal{D}$, we assume that we observe some  real-valued variable $Y$ at $p\ge 1$ locations $\bm x_1, \dots, \bm x_p \in \mathcal{D}$. These observations are modeled as 
\begin{equation*}
	Y({\bm x}_i) = \mathcal{Z}({\bm x}_i) + \tau \epsilon_i, \quad 1 \le i \le p,
\end{equation*}
where $\epsilon_1, \dots, \epsilon_p$ are independent standard Gaussian variables, $\tau >0$, and $\mathcal{Z}$ denotes some GRF on $\mathcal{D}$ acting as a latent variable. Hence, they can be seen as observations of the latent field $\mathcal{Z}$ affected by some independent centered Gaussian noise with variance $\tau^2$.  

We aim at making the Best Linear Unbiased Prediction (BLUP) of the variable $\mathcal{Z}$ at $q$ locations $\bm x_{p+1}, \dots, \bm x_{p+q} \in \mathcal{D}$. Under a Gaussian assumption, recall that the BLUP is equal to the conditional expectation \citep{tong2012multivariate}, i.e. the optimal prediction in a $L^2$-sense. In the geostatistical literature, this prediction is referred to as kriging \citep{chiles1999geost}. In most geostatistical approaches, the GRF is either characterized by a covariance function or by a precision matrix. Here, in contrast, the GRF is defined on a triangulation of $\mc{D}$ and  it is characterized by a positive polynomial $P$ as per \Cref{eq:prec_mat_cov}. In this section, to derive the kriging algorithm, we first suppose that this polynomial is known. We will show in the next section how $\tau^2$ and the coefficients of $P$ can be estimated from a single realization of the vector of observations $\bm Y=(Y(\bm x_1), \dots, Y(\bm x_p))^T$.

Note that the approach we present here readily generalizes to the case where covariates are added to the model. In that case, the vector of observations $\bm Y$ should contain the residuals obtained after removing from the data points the trend defined by the covariates.

\subsection{A \q{Matrix Free} Kriging Algorithm}

We start with some triangulation of $\mathcal{D}$ with $n$ nodes $\bm s_1, \dots, \bm s_n \in \mathcal{D}$, and  $\mathcal{Z}$ is approximated by its finite element approximation $Z$ associated to this triangulation as shown in \Cref{sec:rf_approx}. Following the interpolation rule in~\Cref{eq:z_galer}, the values of the field $Z$ at the observed locations $\bm x_1, \dots, \bm x_p$ can be expressed  as a linear combination of the values taken at the triangulation nodes $\bm s_1, \dots, \bm s_n$. The vector of observations $\bm Y$ can thus be written as
\begin{equation}
	\bm Y = \bm M_\DD \bm Z + \tau \bm\epsilon,
	\label{eq:def_Y}
\end{equation}
where $\bm Z =(Z(\bm s_1), \dots, Z(\bm s_n))^T$, $\bm\epsilon = (\epsilon_1, \dots, \epsilon_p)^T$ and $\bm M_\DD\in\R^{p\times n}$ is the so-called design matrix containing the interpolation weights defined by
\begin{equation}
	[\bm M_\DD]_{ij}=\psi_j(\bm x_i), \quad 1 \le i \le p, \quad 1\le j\le n.
	\label{eq:def_md}
\end{equation}
The next proposition provides an analytic expression for the kriging predictors $Z^*(\bm x_{p+1})$ at some given target location $\bm x_{p+i}\in \mc{D}$ ($1\le i\le q$). It is proven in \Cref{appen:proof_krig}.

\begin{proposition}\label{prop:krig}
	The conditional distribution of $\bm Z$ given $\bm Y$ is that of a Gaussian vector with mean $\e[\bm Z \vert \bm Y]$ and covariance matrix $\cov[\bm Z\vert \bm Y]$ given by
	\begin{equation}
		\e[\bm Z \vert \bm Y]
		=\bm\Sigma \bm M_\DD^{T}(\bm M_\DD \bm\Sigma \bm M_\DD^T +\tau^2\bm I_p)^{-1}
		\bm Y
		=(\tau^2\bm Q + \bm M_\DD^T\bm M_\DD)^{-1}\bm M_\DD^T\bm Y,
		\label{eq:ec}
	\end{equation}
	and
	\begin{equation}
		\cov[\bm Z\vert \bm Y]
		=\bm \Sigma - \bm\Sigma \bm M_\DD^{T}(\bm M_\DD \bm\Sigma \bm M_\DD^T +\tau^2\bm I_p)^{-1}\bm M_\DD\bm \Sigma
		=\tau^2(\tau^2\bm Q + \bm M_\DD^T\bm M_\DD)^{-1},
		\label{eq:vc}
	\end{equation}
	where  $\bm\Sigma$ is the covariance matrix of $\bm Z$, $\bm Q=\bm \Sigma^{-1}$ is its precision matrix, and  $\tau^2$ and $\bm M_D$ are defined in~\Cref{eq:def_Y}. 
	Then, the vector of kriging predictors $\bm Z^*=(Z^*(\bm x_{p+1}), \dots, Z^*(\bm x_{p+q}))^T$ can be written as
	\begin{equation}
		\bm Z^*=\bm M_\TT \e[\bm Z \vert \bm Y],
		\label{eq:krig}
	\end{equation}
	where $\bm M_\TT$ is the target design matrix defined by
	\begin{equation}
		[\bm M_\TT]_{ij} = \psi_j(\bm x_{p+i}), \quad 1\le i\le q, \quad 1\le j\le n.
		\label{eq:def_mt}
	\end{equation}
\end{proposition}

We have assumed that the function $f$  characterizing the field $\mc{Z}$ is the inverse of $P$, a positive polynomial on $\R_+$. This implies in particular that the precision matrix $\bm Q$ in~\eqref{eq:ec} can be expressed as in \Cref{eq:prec_mat_cov}. Assuming that $P$ is known, the  kriging predictors $\bm Z^*$ in~\eqref{eq:krig}  can thus be computed using \Cref{alg:krig}.

\begin{algorithm}[htb]
	\caption{Kriging prediction.}\label{alg:krig}
	\begin{algorithmic}[1]
		\REQUIRE  Matrices $\sqrt{\bm C}$, $\bm S$, $\bm M_\DD$ and $\bm M_\TT$, as defined in \Cref{eq:def_masslump,eq:def_masslump_S,eq:def_md,eq:def_mt}.
		\REQUIRE Polynomial $\displaystyle P$.
		\REQUIRE Parameter $\tau>0$.
		\REQUIRE Vector of observations $\bm Y$. 
		\STATE Solve for $\bm X$ the linear system
		\begin{equation}
			\bm A\bm X =\bm M_\DD^T
			\bm Y
			\label{eq:linsyst}
		\end{equation}
		where 
		\begin{equation}
			\bm A=\tau^2\bm Q + \bm M_\DD^T\bm M_\DD=\tau^2(\sqrt{\bm C})P(\bm S)(\sqrt{\bm C})^{T} + \bm M_\DD^T\bm M_\DD
			\label{eq:def_A}
		\end{equation}
		\STATE Return $\bm Z^* := \bm M_{\TT}\bm X$.
	\end{algorithmic}
\end{algorithm}

Note that all the matrices $\sqrt{\bm C}$, $\bm S$, $\bm M_\DD$ and $\bm M_\TT$ are sparse, and that the matrix $\bm Q$ is also sparse when $P$ has a low degree. A classical approach for solving the linear system in~\Cref{alg:krig} consists in first building and storing the matrix $\bm A=\tau^2\bm Q + \bm M_\DD^T\bm M_\DD$, and then using a method for solving sparse linear systems. For instance, one could use so-called sparse direct solvers. Such solvers start by factorizing $\bm A$ into sparse triangular factors (using LU or Cholesky decompositions) and then solving the resulting triangular systems. The Cholesky decomposition for large $\bm A$ can be done exactly or with very accurate decomposition using Tile Low Rank approximations with the ExaGeoStat software  \citep{abdulah2018exageostat}. Note however, that in this case the matrix $\bm A$ needs to be built and stored.

When these direct approaches are not possible due to size of the problem, an alternative approach using iterative solvers must be used \citep{nocedal2006numerical}.  Such solvers only rely on products between the matrix $\bm A$ and vectors, and therefore can be used to approximately solve the linear system in~\Cref{alg:krig} without effectively requiring to build and store $\bm A$: only a routine performing the product between $\bm A$ and vectors is needed. Note that such products would only require products between the sparse matrices $\sqrt{\bm C}$, $\bm S$ and $\bm M_D$ and vectors (see \Cref{alg:prodA}). This approach yields a so-called \textit{\q{matrix-free}} approach to kriging in the sense that it does not require to explicitly build and store the precision matrix $\bm Q$ of the field.

\begin{algorithm}[htb]
	\caption{Product between a matrix $(\alpha \bm D P(\bm B)\bm D^{T} + \bm M^T\bm M)$ and a vector.}\label{alg:prodA}
	\begin{algorithmic}[1]
		\REQUIRE  Matrices $\bm B, \bm D \in \R^{n\times n}$ and $\bm M\in\R^{p\times n}$
		\REQUIRE Polynomial $\displaystyle P(X)=\sum_{k=0}^Kc_k X^k$ for some $K\in\mathbb{N}_0$, $c_0, \dots, c_K \in\R$
		\REQUIRE Parameter $\alpha\in\R$
		\REQUIRE Vector $\bm v\in\R^n$
		\STATE Compute	$\bm x := \bm M^T \bm M \bm v=\big((\bm M \bm v)^T\bm M\big)^T$
		\STATE 	Compute $\bm w:= \bm D^T\bm v=(\bm v^T\bm D)^T$
		\STATE Compute $\bm y :=c_K \bm w$
		\FOR{$k = K-1,\dots,0$}
		\STATE Compute $\bm y \leftarrow c_{k}\bm w+\bm S \bm y$
		\ENDFOR	
		\STATE Return $\alpha\bm D\bm y + \bm x$
	\end{algorithmic}
\end{algorithm}

\subsection{Conditional Simulations}
\label{sec:sim}

Generating samples from the conditional distribution of $\mathcal{Z}$ (or rather of its approximation $Z$) given $\bm Y$ is now a straightforward task. Since the distribution of $Z$ is entirely specified through the distributions of its weights $\bm Z$,  this amounts to sample from the conditional distribution of $\bm Z$ given $\bm Y$, which will be denoted  $\pi_{\bm Z\vert \bm Y}$.
Recall from~\Cref{prop:krig} that $\pi_{\bm Z\vert \bm Y}$ is a multivariate Gaussian distribution with mean $\e[\bm Z \vert \bm Y]$ and covariance matrix $\cov[\bm Z\vert \bm Y]$ given by \Cref{eq:ec,eq:vc}.
Hence, sampling from the conditional distribution $\pi_{\bm Z\vert \bm Y}$ is straightforward, as shown in in~\Cref{alg:simc}. 

\begin{algorithm}
	\caption{Conditional simulation of the vector $\bm Z$ given the observations $\bm Y$.}\label{alg:simc}
	\begin{algorithmic}[1]
		\REQUIRE  Covariance matrix $\cov[\bm Z\vert \bm Y]$ defined in~\Cref{eq:vc}.
		\REQUIRE Conditional mean $\e[\bm Z \vert \bm Y]$ defined in~\Cref{eq:ec}. 
		\STATE Sample a centered Gaussian vector $\bm X$ with covariance matrix $\cov[\bm Z\vert \bm Y]$.
		\STATE $\bm X \leftarrow \bm X + \e[\bm Z \vert \bm Y]$.
		\STATE Return $\bm X$.
	\end{algorithmic}
\end{algorithm}
The conditional mean $\e[\bm Z \vert \bm Y]$ required for~\Cref{alg:simc} is computed using the same methods as those described in~\Cref{sec:krig} to compute the kriging predictors.

Since an explicit formula is available for $\cov[\bm Z\vert \bm Y]$ in \Cref{eq:vc}, sampling the Gaussian vector of the first step in \Cref{alg:simc} could be directly done by finding a factor $\bm L$ such that $\cov[\bm Z\vert \bm Y]=\bm L\bm L^T$ and then returning the product $\bm L \bm W$ where $\bm W$ is a vector of independent standard Gaussian variables. Possibles candidates for $\bm L$ include the Cholesky decomposition of $\cov[\bm Z\vert \bm Y]$, but also the matrix function $h(\cov[\bm Z\vert \bm Y])$ where $h$ is the square-root function, and the matrix function $\tilde{h}(\tau^2\bm Q + \bm M_\DD^T\bm M_\DD)$ where $\tilde{h}$ is the inverse square-root function. In practice, and as before, both matrix functions could be approximated by matrix polynomials and the \q{matrix-free} algorithms presented above could be used.

Another method to sample the Gaussian vector of the first step in \Cref{alg:simc} stems from the fact that the vector  $\bm Z - \e[\bm Z \vert \bm Y]$ is such a vector, and that the expression of $\cov[\bm Z\vert \bm Y]$ does not explicitly depend on the vectors $\bm Z$ and $\bm Y$. Hence, by sampling new vectors $\bm Z'$ and $\bm Y'$ with the same distribution as $\bm Z$ and $\bm Y$, and by computing $\bm Z' - \e[\bm Z' \vert \bm Y']$ we retrieve a centered Gaussian vector with covariance matrix $\cov[\bm Z'\vert \bm Y']=\cov[\bm Z\vert \bm Y]$. This method is presented in \Cref{alg:sim_cov_cond}. It relies on being able to sample the vector $\bm Z'$ described above, which can be done by either:
\begin{itemize}
	\item Computing the product $\bm B \bm W$ where $\bm B$ is a square-root of $\bm\Sigma$ (e.g. the Cholesky factor of $\bm \Sigma$ or the matrix $(\sqrt{\bm C})^{-T}(1/\sqrt{P})(\bm S)$).
	\item Solving for $\bm X$ the linear system  $\tilde{\bm B}\bm X= \bm W$, where  $\tilde{\bm B}$ is a square-root of $\bm Q$ (e.g. the Cholesky factor of $\bm Q$ or the matrix $(\sqrt{\bm C}){\sqrt{P}}(\bm S)$).
\end{itemize}
In practice, the matrix functions appearing above are once again polynomially approximated.
More details on the step 1 of Algorithm \ref{alg:sim_cov_cond} can be found in \cite{pereira2019efficient}.

\begin{algorithm}
	\caption{Sampling a centered Gaussian vectors with covariance matrix $\cov[\bm Z\vert \bm Y]$.}\label{alg:sim_cov_cond}
	\begin{algorithmic}[1]
		\REQUIRE  Covariance matrix $\bm \Sigma$ in \Cref{eq:cov_mat_g} or precision matrix in \Cref{eq:prec_mat_cov}.
		\REQUIRE Matrix $\bm M_\DD$ and parameter $\tau$ defining the observations $\bm Y$ in \Cref{eq:def_Y}. 
		
		\STATE Sample a centered Gaussian vector $\bm Z'$ with covariance matrix $\bm \Sigma$ or precision matrix $\bm Q$.
		
		\STATE Sample a vector $\bm \epsilon'$ with independent standard Gaussian entries.
		
		\STATE Compute $\bm Y' :=\bm M_\DD \bm Z' + \tau \epsilon'$.
		
		\STATE  Compute $\e[\bm Z' \vert \bm Y']$ by replacing $\bm Y$ by $\bm Y'$ in \Cref{eq:ec}.
		
		\STATE  Return  $\bm Z' - \e[\bm Z' \vert \bm Y']$.
	\end{algorithmic}
\end{algorithm}

\section{Estimation of the Parameters}
\label{sec:estim}

In order to apply the kriging and simulation procedures presented above, one needs to know the polynomial $P$  characterizing the field $\mathcal{Z}$, as well as the parameter $\tau^2$ defining the variance of the Gaussian noise. Usually, these parameters are not known, and they must be estimated from the observations $\bm Y$. We now show how a maximum likelihood approach can be efficiently implemented. Let us denote by $\bm\theta$  the vector of parameters containing the coefficients of $P$ and the variance $\tau^2$. Note then that, following \Cref{eq:def_Y}, the vector $\bm Y$ is  a centered Gaussian vector with covariance matrix
\begin{equation}
	\bm\Sigma_{\bm Y}(\bm\theta)=\bm M_D \bm Q(\bm\theta)^{-1}\bm M_D^T + \tau^2 \bm I_p,
\end{equation}
where $\bm Q(\bm\theta)=(\sqrt{\bm C})P(\bm S)(\sqrt{\bm C})^{T}$.
The log-likelihood of $\bm Y$ is thus 
\begin{equation*}
	\mc{L}(\bm\theta)= -\frac{1}{2}\big(p\log 2\pi -\log \vert\bm Q_{\bm Y}(\bm\theta)\vert + \bm Y^T\bm Q_{\bm Y}(\bm\theta)\bm Y \big),
\end{equation*}
where $\bm Q_{\bm Y}(\bm\theta)=\bm\Sigma_{\bm Y}(\bm\theta)^{-1}$. Considering the matrix  $\bm A=\bm A(\bm\theta)$ defined in~\eqref{eq:def_A}, we can write 
\begin{equation}
	\bm Y^T\bm Q_{\bm Y}(\bm\theta)\bm Y= \tau^{-2}\big(\bm Y^T\bm Y-\bm Y^T\bm M_D\bm A(\bm\theta)^{-1}\bm M_D^T\bm Y\big),
	\label{eq:quad_form}
\end{equation}
and
\begin{equation}
	\begin{aligned}
		\log \vert\bm Q_{\bm Y}(\bm\theta)\vert
		&= \log \vert\bm Q(\bm\theta)\vert+(n-p)\log \tau^2-\log\vert \bm A(\bm\theta)\vert \\
		&=\log \vert P(\bm S) \vert +2\log\vert\sqrt{\bm C}\vert+(n-p)\log \tau^2-\log\vert \bm A(\bm \theta) \vert.
	\end{aligned}
	\label{eq:logdet}
\end{equation}
The precise computation of the log-likelihood deserves some comments. The quadratic form~\eqref{eq:quad_form} is computed  using the methods  introduced to solve the linear system, as shown in \Cref{alg:krig} and in \Cref{eq:linsyst}. Evaluating the log-determinants in~\Cref{eq:logdet} could be done using a Cholesky (or LU) factorization of the matrices $\bm Q(\bm \theta)$ and $\bm A(\bm\theta)$ and then summing the log of the diagonal elements of the resulting triangular factors. However, one can also use a \q{matrix-free} approach to compute an approximation of these log-determinants. This approach, detailed in the rest of this section, is a generalization to matrix functions of the results in~\citet{han2015large}. It is based on the following result (proven in \Cref{appen:proof_logdet}).

\begin{proposition}
	\label{prop:logdet}
	Let $\bm B$ be a diagonalizable matrix and $h : \R \rightarrow (0, \infty)$. The log-determinant of the matrix function $h(\bm B)$ satisfies the relation
	\begin{equation*}
		\log \vert h(\bm B) \vert = \trace(\log h(\bm B))=\e[\bm W ^T \log h(\bm B) \bm W],
	\end{equation*}
	where $\bm W$ is a vector whose entries are independent zero-mean unit-variance random variables, and $\log h(\bm B)$ is the matrix function defined from the function $\log h$.
\end{proposition}

The matrix function $\log h(\bm B)$ can in practice be approximated by a matrix polynomial $P_{\log h}(\bm B)$ where $P_{\log h}$ denotes a polynomial approximation of $\log h$ over an interval containing the eigenvalues of $\bm B$ (and defined using for instance Chebyshev polynomial approximation). Then, we can approximate $\log \vert h(\bm B)\vert$ as
\begin{equation}
	\log \vert h(\bm B)\vert \approx \frac{1}{M}\sum_{m=1}^M \bm W_m ^T P_{\log h}(\bm B) \bm W_m,
	\label{eq:approx_logdet}
\end{equation}
where $\bm W_1, \dots, \bm W_M$ denote $M$ independent samples of $\bm W$ (defined for instance from a Gaussian or Rademacher distribution).  Similarly to \Cref{alg:prodA}, each quadratic form in \Cref{eq:approx_logdet} is computed in an iterative way while only requiring products between the matrix $\bm B$ and vectors. This approach can  be used to compute the log-determinant~\eqref{eq:logdet} by noting that both $\log \vert P(\bm S)\vert$ and $\log\vert \bm A(\bm\theta)\vert$ can be written in the form $\log h(\bm B)$ with $\bm B=\bm S$ and $h=P$ for the former, and $\bm B=\bm A(\bm\theta)$ and $h$ to be the identity map for the latter. The whole procedure to compute  $\log \vert\bm Q_{\bm Y}(\bm\theta)\vert$ is summarized in \Cref{alg:logdet}. 
Note in particular that this algorithms requires to know the intervals containing the eigenvalues of $\bm S$ and those of $\bm A(\bm \theta)$. Details on the computation of these intervals can be found in \Cref{appen:interval}. A discussion about the \q{matrix-free} approach to the computation of $\log | h(\bm B)|$ is deferred to \Cref{sec:disc}.

\begin{algorithm}
	\caption{Computation of  $\log \vert\bm Q_{\bm Y}(\bm\theta)\vert$ defined in \Cref{eq:logdet}}
	\label{alg:logdet}
	\begin{algorithmic}[1]
		\REQUIRE  Matrices $\sqrt{\bm C}, \bm S \in\R^{n\times n}$ and $\bm M_\DD \in\R^{p\times n}$  as defined in \Cref{eq:def_masslump,eq:def_masslump_S,eq:def_md}
		\REQUIRE Vector $\bm\theta$ containing the coefficients of $P$ and the variance parameter $\tau^2$
		\REQUIRE Number of samples $M$
		\STATE Compute the coefficients of a polynomial approximation $P_{\log P}$ of the function $\lambda \mapsto \log P(\lambda)$ over an interval containing the eigenvalues of $\bm S$
		\STATE Compute the coefficients of a polynomial approximation $P_{\log}$ of the function $\lambda \mapsto \log \lambda$ over an interval containing the eigenvalues of the matrix $\bm A(\theta)$ defined in \Cref{eq:def_A}
		\STATE Set $Q_1=Q_2:=0$
		\FOR{m=1,\dots,M}
		\STATE Sample a vector $\bm W$ with  independent identically distributed entries with mean $0$ and variance $1$ 
		\STATE Compute $\bm u:= P_{\log P}(\bm S) \bm W$ using \Cref{alg:prodA} 
		\STATE Set $Q_1 \leftarrow Q_1 + \bm W^T \bm u$ 
		\STATE Compute $\bm v:= P_{\log}(\bm A (\bm \theta)) \bm W$ using \Cref{alg:prodA} 
		\STATE Set $Q_2 \leftarrow Q_2 + \bm W^T \bm v$ 
		\ENDFOR
		\STATE Compute $L = (Q_1-Q_2)/M + (n-p) \log \tau^2 + 2\sum_{i=1}^n \log [\sqrt{\bm C}]_{ii}$ 
		\RETURN $L$
	\end{algorithmic}
\end{algorithm}

Since we can now evaluate the log-likelihood for any vector of parameters, we can plug \Cref{alg:logdet} into any optimization algorithm that only requires evaluations of an objective function to maximize it. Examples of such algorithms include the Nelder-Mead algorithm, and any gradient-descent algorithm for which the gradients would be numerically approximated by finite differences \citep{nocedal2006numerical}. 

Finally, note that the procedure presented in this section naturally extends to the case where covariates are added to the model. Indeed, assuming now that the observations are modeled as $\tilde{\bm Y}={\bm X}\bm \beta + \bm Y$ where $\bm Y$ is defined as before, $\bm X$ is a matrix of covariates and $\bm \beta$ a vector containing the associated regression coefficients. Then, the maximum likelihood estimate of $\bm\beta$ is given by $\widehat{\bm \beta}(\bm\theta)=(\bm X^T \bm Q_{\bm Y}(\bm\theta)\bm X)^{-1}\bm X\bm Q_{\bm Y}(\bm\theta) \bm \tilde{\bm Y}$, and the estimation of the parameter can be carried out by maximizing the likelihood $\mc{L}$, where $\bm Y$ is now replaced by $\tilde{\bm Y}-\bm X\widehat{\bm\beta}(\bm\theta)$.

\section{Illustration}
\label{sec:illustration}
The proposed algorithms are illustrated on synthetic very large data sets. Our aim is to show that our approach compares very well with the GMRF approximation in \citet{lindgren2011explicit} on very large grids, even in cases that are favorable to the GMRF approximation.

\subsection{Simulation and Kriging}

We first consider the case of 3D standardized GRF with varying anisotropies in the horizontal plane and an exponential covariance, which corresponds in 3D to the choice $P(\lambda)=(\kappa^2 +\lambda)^2$. Following the remarks of \Cref{sec:sim}, it is sampled by using a Chebyshev polynomial approximation of $\lambda\mapsto 1/\sqrt{P}(\lambda)=(\kappa^2+\lambda)^{-1}$ of degree 268 \citep[see e.g.][for details]{pereira2019efficient}. We chose a mesh built from 6 tetrahedrons in each cell of a regular grid with lags suitably chosen to fit with the ranges of the GRF. The resulting size of the discretized vector $\bm Z$ is about $1.5\times 10^7$. On this extremely large grid, the simulation takes only 30 seconds on a laptop running at 1.9 Ghz on 8 cores. Most of the time is spent for the matrix-vector products implied by the polynomial approximation as shown in \Cref{alg:prodA}.

A sub-sampling of this simulation is done to obtain $10^5$ randomly located observations which are used to perform kriging. Since interpolation by kriging is smoother than a simulation, a coarser mesh can be used. The size of the resulting kriging system is about $7 \times 10^6$.  A measurement error with a variance $\tau^2=0.01$ is added to the model. Conjugate gradient without preconditioning is used to solve the system of \Cref{eq:linsyst} in  \Cref{alg:krig}. The algorithm converges in 1098 iterations for a computing time around 400 seconds. Results are displayed in \Cref{fig:3dblock}.
Note that in most applications, the variance $\tau^2$ is greater than $1\%$ of the variance. In these cases, the system of \Cref{eq:linsyst} is better conditioned, leading to a faster convergence of the conjugate gradient.
Kriging is done here in a nonstationary context, with a stationary mesh which is thus far from optimality in most locations. If the simulation were stationary, the mesh could be tailored to the specific model at hand. The grid could be oriented according to the anisotropy tensor and the lag in each direction could be chosen according to the associated directional range of the model. As a result, a given accuracy would be obtained with a lower degree of the Chebyshev polynomial and system \Cref{eq:linsyst} would be better conditioned, thus leading to faster computations. In our experience, stationary simulations and kriging are usually 10 times faster than nonstationary ones, all other settings being equal.

\begin{figure}[thb]
	\centering
	\includegraphics[width=0.4\textwidth]{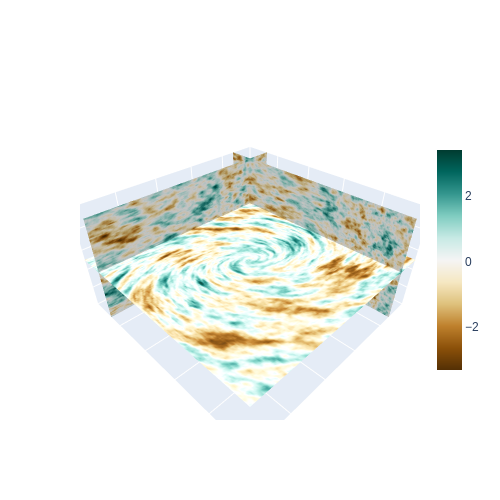}
	\includegraphics[width=0.4\textwidth]{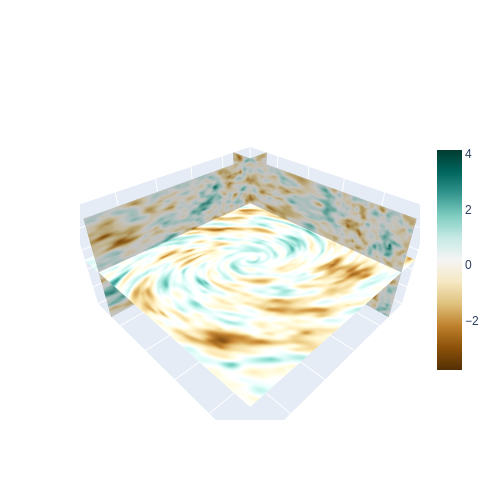}
	\caption{Left:  3D  simulation of a GRF with varying anisotropies. Right: Kriging estimate using $10^5$ randomly located samples from the simulation on the left.}\label{fig:3dblock}
\end{figure}

\subsection{Estimation of the Parameters of a GMRF}
\label{sec:mle}

In this example, the coefficients of a polynomial $P$ characterizing an isotropic GMRF and the measurement error $\tau^2$ are estimated by the approach described in \Cref{sec:estim}. The coefficients of the true $P$ are given in \Cref{tab:table_param}. It does not correspond to a Matérn type. $1.5 \times 10^4$ synthetic observations are sampled from the model at random locations of a square of size $40$ and a Gaussian noise with variance $\tau^2=0.01$ is added (see \Cref{fig:simu2d}).

\begin{figure}[thb]
	\centering
	\includegraphics[width=0.335\textwidth]{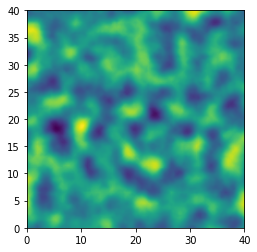}
	\includegraphics[width=0.495\textwidth]{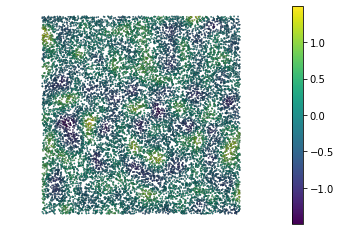}
	\caption{Left:  Realization of the GMRF on the square. Right: Sampling of $1.5 \times 10^4$ noisy observations from the simulation on the left.}\label{fig:simu2d}
\end{figure}

To ensure that $P$ takes strictly positive values over $\R_+$ during the estimation, the problem is re-parameterized by using two arbitrary polynomials $P_1$ and $P_2$ as follows:

\begin{equation}
	P(x) =  P_1^2(x) + x P_2^2(x).
	\label{eq:param_positive_poly}
\end{equation}

Indeed, all the positive polynomials over $\R_+$ can be written according to \Cref{eq:param_positive_poly}. 
A fixed positive value $\varepsilon$ is added to the right-hand side of \Cref{eq:param_positive_poly} to ensure strict positivity of $P$.
In this study, $\varepsilon$ is set to $10^{-3}$.
The components of $\bm \theta$ are the coefficients of $P_1$ and $P_2$. The likelihood is maximized with the COBYLA algorithm \citep{powell1994direct}, using several initial guesses and selecting the best output. To compute the log-determinant by the Hutchinson estimators \citep{hutchinson1989stochastic}, the same Gaussian vectors are used for all the iterations. First only one vector is used in order to find an acceptable solution. Then the algorithm is relaunched and 10 random vectors are used until convergence.

The true coefficients of $P$ (resulting from those of $P_1$ and $P_2$) are given in \Cref{tab:table_param} along with the estimated coefficients. We also give in this table the initial guess that yielded the estimated coefficients. The initial value for $\tau^2$ was set to $1$ and its estimation is $0.00944$.
At first glance, the estimation results does not seem very accurate.
However, the associated covariance functions, computed by Fast-Fourier-Transform (FFT) and displayed on \Cref{fig:cova_result} lead to a more positive conclusion. Indeed, the estimated covariance is close to the true one, especially near the origin. Nevertheless, the likelihood seems to have several modes and the results are sensitive to the choice of the initial values. $\tau^2$ is generally well estimated but the estimated covariance of the underlying GMRF is not always so close to the true one. Despite this optimization problem, further discussed in the next section, the matrix-free approach allows to approximate the likelihood in order to estimate the shape of the covariance, even with a moderate number of random vectors for the Hutchinson estimator.

\begin{table}
	\caption{Coefficients of the polynomials (true, estimated, and initial).}    \label{tab:table_param}
	\centering
	\begin{tabular}{rcccc}
		\hline \hline
		Degree & 0 & 1 & 2 & 3 \\
		\hline
		True & 1 & -0.75 & -0.75 & 1\\
		Estimated & 0.55 & 2.42 &-5.17 &2.60 \\
		Initial & 0.001 & 0 & 0 & 0 \\
		\hline \hline
	\end{tabular}
\end{table}

\begin{figure}[thb]
	\centering
	\includegraphics[width=0.49\textwidth]{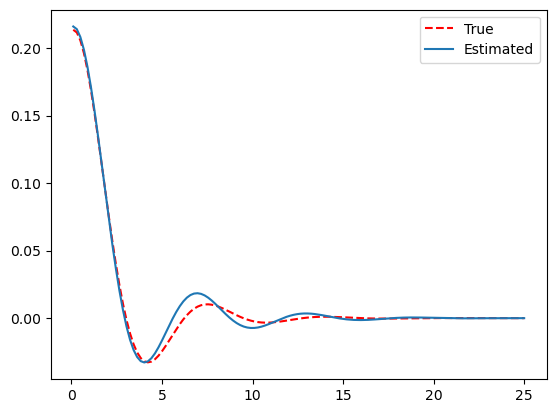}
	\caption{Covariances obtained from $P$ by FFT (true model and estimated).}\label{fig:cova_result}
\end{figure}
\section{Discussion and Conclusion}
\label{sec:disc}

We have proposed a generic approach for defining GRFs on compact Riemannian manifolds. It combines two ingredients of modern geostatistics, as pioneered in \citet{lindgren2011explicit}: the definition of GRFs through the expansions in the eigenfunctions of the Laplace–Beltrami operator on the Riemannian manifold and the finite element approximation of these GRFs. This approach is quite general. Not limited to spheres, it can be applied to construct valid GRFs on any smooth compact manifold -- and in particular to any compact surface. Moreover, since the GRF is characterized by its spectral density, it is not limited to Whittle-Matérn random fields.  The proposed Riemannian metric offers a straightforward interpretation of local anisotropies. Our approach is thus also perfectly suited to the analysis of data in Euclidean domains with nonstationary range and nonstationary anisotropies.

For this quite general class of GRFs, we have provided efficient algorithms that do not require to build and store possibly very large matrices. Instead, our \q{matrix-free} approach is grounded on algorithms only requiring efficient routines for computing the product between very sparse matrices and vectors, as shown in \Cref{alg:prodA}. This ensures the scalability of this method, thus paving the way to efficient nonstationary geostatistics for large datasets. We have shown on synthetic examples that our approach is able to handle grids with millions of nodes (up to $10^7$) in only few minutes. 

The main bottleneck is the computation of the log-determinant of the matrix function $h(\bm B)$ in \Cref{prop:logdet}. When possible, the Cholesky decomposition of $\bm B$ is best. This is the case for matrices whose size is in the range of $10^4$ or smaller. The \q{matrix-free} approach described above scales very well with the size of $\bm B$. It has been successfully applied in the context of seismic filtering \citep{pereiraPhd2019,pereira2020matrix} and it is now part of industrial codes able to filter very large noisy seismic datasets (in the range $10^6$ to $10^7$ in few minutes). Even though being scalable, our algorithms could be accelerated in some situations. \Cref{alg:logdet} requires a Monte-Carlo loop relating to the Hutchinson estimator \Cref{eq:approx_logdet} and iterated products as described in \Cref{alg:prodA} which can be long if the degree of the polynomial is high. Further research is needed in order to better assess the precision of the approximation \eqref{eq:approx_logdet} in view of optimizing the computation of the log-determinant. 

Besides, the inference of parameters through likelihood maximization poses some challenges due to the existence in general of local maxima \citep{williams2006gaussian}. In our case, we performed the minimization with several initial guesses and chose the one yielding the best result. In order to robustify the estimation, one could move to algorithms better equipped to handle local maxima, e.g. particle swarm optimizers and other evolutionary optimizers.  Inspired by the recent advances in Machine Learning and Deep Learning, another line of research is how to maximize the log-likelihood using (stochastic) gradient algorithms \citep{simon2013evolutionary}.

We believe that the \q{matrix-free } approach is thus a very promising tool for the analysis of environmental dataset that will be tested in future work, including in a spatio-temporal setting. Other directions for future research include 3D extensions and the modeling of nonstationary anisotropies on spheres and on manifolds in general.

\section*{Supplementary Material}

The code used to perform the maximum likelihood estimation in \Cref{sec:mle} is available at \url{https://github.com/mike-pereira/matrix-free-mle}.

\section*{Acknowledgements} The authors acknowledge the support of the Mines Paris / INRAE chair \q{Geolearning}.

\section*{Competing Interests}
The authors have no competing interests to declare.

\bibliographystyle{jds}       
\bibliography{bib}

\appendix

\section{Proofs}

\subsection{Proof of Proposition \ref{prop:krig}}
\label{appen:proof_krig}

Consider the vector $\bm X$ given by
\begin{equation*}
	\bm X= 
	\begin{pmatrix}
		\bm Z \\
		\bm Y
	\end{pmatrix}
	=
	\begin{pmatrix}
		\bm I_n &  \bm 0_{np}\\
		\bm M_\DD & \tau\bm I_p
	\end{pmatrix}
	\begin{pmatrix}
		\bm Z \\
		\bm \epsilon
	\end{pmatrix}.
\end{equation*}
Note that $\bm X$ is a centered Gaussian vector with covariance matrix
\begin{equation}
	\begin{aligned}
		\cov[\bm X]
		&=
		\begin{pmatrix}
			\bm I_n &  \bm 0_{np}\\
			\bm M_\DD & \tau\bm I_p
		\end{pmatrix}
		\begin{pmatrix}
			\bm \Sigma & \bm 0_{nn} \\
			\bm 0_{pp} & \bm I_p
		\end{pmatrix}
		\begin{pmatrix}
			\bm I_n &  \bm 0_{np} \\
			\bm M_\DD & \tau\bm I_p
		\end{pmatrix}^T
		=		\begin{pmatrix}
			\bm \Sigma &  \bm\Sigma \bm M_\DD^T\\
			\bm M_\DD\bm \Sigma & \bm M_\DD\bm\Sigma \bm M_\DD^T+\tau^2\bm I_p
		\end{pmatrix}.
	\end{aligned}
	\label{eq:covX}
\end{equation}
Since $\bm X$ is multivariate Gaussian it follows that  the conditional distribution of $\bm Z$ given $\bm Y$ is a Gaussian vector with mean $\e[\bm Z \vert \bm Y]$ and covariance matrix $\cov[\bm Z\vert \bm Y]$  \citep[Theorem 3.3.4]{tong2012multivariate} given by
\begin{equation*}
	\e[\bm Z \vert \bm Y]
	=\bm\Sigma \bm M_\DD^{T}(\bm M_\DD \bm\Sigma \bm M_\DD^T +\tau^2\bm I_p)^{-1}
	\bm Y,
\end{equation*}
\begin{equation*}
	\cov[\bm Z\vert \bm Y]
	=\bm \Sigma - \bm\Sigma \bm M_\DD^{T}(\bm M_\DD \bm\Sigma \bm M_\DD^T +\tau^2\bm I_p)^{-1}\bm M_\DD\bm \Sigma.
\end{equation*}
Then, \Cref{eq:ec,eq:vc} follow from computing $\cov[\bm X]^{-1}$ from~\Cref{eq:covX} and using the fact that $\cov[\bm X]\cov[\bm X]^{-1}=\bm I_{p+n}$.

Finally, recall that since we are dealing with Gaussian vectors, the vector of kriging predictors $\bm Z^*$ coincides with the conditional expectation of the vector $\bm Z_\TT=(Z(x_{p+1}), \dots, Z(x_{p+q}))^T$, given the observations $\bm Y$. Hence, by linearity of the expectation and definition of $\bm M_\TT$, we have
\begin{equation*}
	\bm Z^* = \e[\bm Z_\TT \vert \bm Y] = \e[\bm M_\TT \bm Z \vert \bm Y] =\bm M_\TT \e[\bm Z \vert \bm Y].
\end{equation*}

\vspace{-1em}
\hfill \ensuremath{\Box}

\subsection{Proof of Proposition \ref{prop:logdet}}\label{appen:proof_logdet}

By definition of matrix functions, and denoting by $\lambda_1, \dots, \lambda_n$ the eigenvalues of $\bm B$, we have
\begin{equation*}
	\trace(\log h(\bm B))
	=\sum_{i=1}^n \log h(\lambda_i)=\log\prod_{i=1}^n  h(\lambda_i)=\log \vert h(\bm B) \vert. 
\end{equation*}
The second equality is then a direct consequence of the well-known Hutchinson trace estimator \citep{hutchinson1989stochastic}.
\hfill \ensuremath{\Box}

\section{Intervals of Eigenvalues}\label{appen:interval}
We here show how  intervals containing all eigenvalues of the matrices $\bm S$ and $\bm A$ respectively defined in \Cref{eq:def_S} and \Cref{eq:def_A} can be computed.  Since the matrix $\bm S$ is positive definite, an interval containing its eigenvalues is obtained by considering $[0, \lambda_{\max}(\bm S)]$, where $\lambda_{\max}(\bm S)$ is an upper-bound of the eigenvalues of $\bm S$. This upper-bound can be obtained  by taking $\lambda_{\max}(\bm S)=\sqrt{\trace(\bm S^T\bm S)}$ or by relying on Gershgorin circle theorem \citep{gershgorin1931uber}, which we now recall.

\begin{proposition}
	\label{prop:eigenvalues}
	The eigenvalues of a symmetric matrix $\bm B\in\R^{n\times n}$ with entries $B_{ij}$ are contained in the interval 
	\begin{equation*}
		[\min_{1\le i\le n}(B_{ii}-R_i), \max_{1\le i\le n}(B_{ii}+R_i)],
	\end{equation*}
	where $R_i=\sum_{j\neq i} \vert B_{ij}\vert$, $1\le i\le n$. 
\end{proposition}

Regarding the matrix $\bm A = \tau^2\bm Q + \bm M_\DD^T\bm M_\DD$, lower and upper bounds of its eigenvalues are given in the next proposition.
\begin{proposition}\label{prop:gersh}
	Let $\lambda_{\min}(\bm A)$ (resp. $\lambda_{\max}(\bm A)$) denote some lower (resp. upper) bound of the eigenvalues of the matrix $\bm A$. Then,
	\begin{align*}
		\lambda_{\min}(\bm A) &= \tau^2\bigg(\min_{1\le i\le n} [\sqrt{\bm C}]_{ii}^2\bigg)\bigg(\inf_{\lambda\in [0,\lambda_{\max}(\bm S)]}P(\lambda)\bigg), \\
		\lambda_{\max}(\bm A)& = \tau^2\bigg(\max_{1\le i\le n} [\sqrt{\bm C}]_{ii}^2\bigg)\bigg(\sup_{\lambda\in [0,\lambda_{\max}(\bm S)]}P(\lambda)\bigg) + \bigg(\max_{1\le i\le n}\sum_{k=1}^p [\bm M_\DD]_{ki}\bigg).
	\end{align*}
\end{proposition}

\begin{proof}
	To ease the notation, let $\bm B=\bm M_\DD^T\bm M_\DD$. Note that we can take
	\begin{equation*}
		\lambda_{\min}(\bm A)= \tau^2\lambda_{\min}(\bm Q) + \lambda_{\min}(\bm B), \quad
		\lambda_{\max}(\bm A)= \tau^2\lambda_{\max}(\bm Q) + \lambda_{\max}(\bm B),
	\end{equation*}
	where $\lambda_{\min}(\cdot)$ (resp. $\lambda_{\max}(\cdot)$) denotes some lower (resp. upper) bound of the eigenvalues of a matrix. On the one hand, since the matrix $\sqrt{\bm C}$ is diagonal, we can take $\lambda_{\min}(\bm Q)=\min_{1\le i\le n} [\sqrt{\bm C}]_{ii}^2 \lambda_{\min}(P(\bm S))$ and $\lambda_{\max}(\bm Q)=\max_{1\le i\le n} [\sqrt{\bm C}]_{ii}^2 \lambda_{\max}(P(\bm S))$. Recalling then the definition of matrix functions, it is clear that the eigenvalues of $P(\bm S)$ are lower (resp. upper) bounded by the infimum (resp. supremum) of $P$ over an interval containing the eigenvalues of $\bm S$, e.g. $[0,\lambda_{\max}(\bm S)]$.
	
	Finally, noting that $\bm B$ is positive semi-definite, we can take $\lambda_{\min}(\bm B)=0$. Then, \Cref{prop:eigenvalues} and the non-negativity of the entries of $\bm B$ allow us to  get $\lambda_{\max}(\bm B)=\max_{1\le i\le n}(B_{ii}+R_i)$
	where
	\begin{equation*}
		B_{ii}+R_i=B_{ii}+\sum_{j\neq i} B_{ij}=\sum_{j=1}^n B_{ij}=\sum_{j=1}^n\sum_{k=1}^p [\bm M_\DD]_{ki}[\bm M_\DD]_{kj}
		=\sum_{k=1}^p [\bm M_\DD]_{ki}\sum_{j=1}^n[\bm M_\DD]_{kj}.
	\end{equation*}
	Noting then that the rows of $\bm M_\DD$ sum to $1$ (since they correspond to linear interpolation weights), we have $B_{ii}+R_i
	=\sum_{k=1}^p [\bm M_\DD]_{ki}$,
	which ends the proof.
\end{proof}

\section{Integration on Riemannian Manifolds}
\label{appen:rm}
We recall usual formulas related to the computation of integrals defined over Riemannian manifolds. We refer the reader to \citet{jost2008riemannian,lee2013smooth} for further. Let $(\mathcal{D}, g)$ be a compact Riemannian manifold of dimension $d$. Let $(U, \phi)$ denote a coordinate chart of $\mc{D}$, i.e. $U$ is an open subset of $\mathcal{D}$ and $\phi$ is a homeomorphism mapping $U$ to an open subset of $\R^d$. Recall that the integral of a function $f : \mathcal{D} \rightarrow \R$ over $U$ is defined as the quantity
\begin{equation*}
	\int_{U} f dV_g = \int_{\phi(U)} f\circ \phi^{-1}(\bm x) \sqrt{\vert g\vert (\phi^{-1}(\bm x))} d\bm x,
\end{equation*}
where $\vert g\vert (\phi^{-1}(\bm x))$ is the determinant of the metric tensor of $g$ at $\phi^{-1}(\bm x)\in U$, expressed in the coordinate chart $(U,\phi)$. The integral of $f$ over $\mathcal{D}$ is then obtained by gluing together (using a partition of unity) local integrals over a collection of coordinate charts that cover $\mc{D}$.

In particular, assuming now that we have a triangulation of $\mathcal{D}$, the integral of $f$ over $\mc{D}$ can be obtained by summing local integrals over each triangle $T$ of the triangulation. In this case, the diffeomorphism $\phi$ associated to $T$ is the map that sends $T$ to the standard simplex of $\R^d$.

Let $L^2(\mc{D})$ denote the set of square-integrable functions of $(\mc{D},g)$. $L^2(\mc{D})$ is a Hilbert space when equipped with the inner product $(\cdot, \cdot)$ defined by $(f,g)=\int_{\mc{D}} fg dV_g$. Note that for any differentiable functions $f_1$ and $f_2$ we denote by $(\nabla f_1, \nabla f_2)$ the integral over $\mc{D}$ of the function $h: p \mapsto g_p(\nabla f_1(p), \nabla f_2(p))$.

And in turn, given a coordinate chart $(U, \phi)$ of $\mathcal{D}$, the integral of $h$ over $U$ reduces to 
\begin{equation*}
	\int_{U} h dV_g = \int_{\phi(U)} \nabla_{\R^d}(f_1\circ\phi^{-1})(\bm x)^T \bm G(\phi^{-1}(\bm x))^{-1} \nabla_{\R^d}(f_2\circ\phi^{-1})(\bm x)  \sqrt{\vert g\vert (\phi^{-1}(\bm x))} d\bm x,
\end{equation*}
where $\bm G(\cdot)$ denotes the metric tensor at given point of $\mathcal{D}$ and expressed in the coordinate chart $(U, \phi)$,  and $\nabla_{\R^d}$ denotes the usual gradient of functions of $\R^d$.

\section{Galerkin Approximation}
\label{appen:galerkin_approx}
Let $\psi_1, \dots, \psi_n$ denote $n$ linearly independent functions from $\mc{D}$  to $\R$ and let $V_n$ denote their linear span. The Galerkin approximation $-\Delta_n$ of the Laplace--Beltrami operator $-\Delta$ is the endomorphism mapping any $f\in V_n$ to the element $-\Delta_n f \in V_n$ satisfying for any $u\in V_n$, $( -\Delta_n f, u)=( -\Delta f, u)$. This endomorphism is diagonalizable, and shares the same eigenvalues as the scaled stiffness matrix $\bm S$ defined in~\Cref{eq:def_S} \citep{lang2021galerkin}.

\end{document}